\documentclass[letterpaper, 10 pt, conference]{ieeeconf}  

\IEEEoverridecommandlockouts                              

\usepackage{cite}
\usepackage{amsfonts}
\usepackage{algorithmic}
\usepackage{graphicx}
\usepackage [autostyle, english = american]{csquotes}
\MakeOuterQuote{"}

\usepackage{amsmath} 
\usepackage{amssymb}

\usepackage{amsthm}  

\usepackage{mathrsfs}
\newcommand{\eat}[1]{}
\usepackage{todonotes}

\newcommand{\Rmnum}[1]{\MakeUppercase{\romannumeral #1}}
\newtheorem{theorem}{Theorem}
\newtheorem{lemma}{Lemma}
\newtheorem{definition}{Definition}

\newtheorem{remark}{Remark}

\newtheorem{proposition}{Proposition}

\newtheorem{assumption}{Assumption}
\newtheorem{problem}{Problem}

\title{
\textbf{An Optimal Coordination Framework for Connected and Automated Vehicles in two Interconnected Intersections}
}

\author{Behdad Chalaki, {\itshape{Student Member, IEEE}}, Andreas A. Malikopoulos, {\itshape{Senior Member, IEEE}}
\thanks{This research was supported in part by ARPAE's NEXTCAR program under the
award number DE-AR0000796 and by the  Delaware Energy Institute (DEI).}%
\thanks{Behdad Chalaki and Andreas A. Malikopoulos are with the Department of Mechanical Engineering, University of Delaware, Newark, DE 19716 USA (email:  \texttt{bchalaki@udel.edu};  \texttt{andreas@udel.edu}.) }%
}
\begin{document}
\maketitle
\thispagestyle{empty}
\pagestyle{empty}
\begin{abstract}
In this paper, we provide a decentralized optimal control framework for coordinating connected and automated vehicles (CAVs) in two interconnected intersections. We formulate a control problem and provide a solution that can be implemented in real time. The solution yields the optimal acceleration/deceleration of each CAV under the safety constraint at "conflict zones," where there is a chance of potential collision. Our objective is to minimize travel time for each CAV. If no such solution exists, then each CAV solves an energy-optimal control problem. 
We evaluate the effectiveness of the efficiency of the proposed framework through simulation.

\end{abstract}

\section{Introduction}

Due to the increasing population and evolution of lifestyle, traffic congestion has become a significant concern in big metropolitan areas. By 2050, it is expected that \(66\%\) of the population will reside in urban areas. By 2030 there will be 41 Mega-cities (with more than 10M people) \cite{Cassandras2017}. Schrank et al. \cite{schrank20152015} reported that in 2014 traffic congestion in urban areas in the US caused drivers to spend 6.9 billion additional hours on the road, burning 3.1 billion extra gallons of fuel.

One of the promising ways to address traffic congestion is to make cities integrated  with information and communication technologies. Using CAVs is one of the intriguing ways towards making smart cities \cite{Klein2016a},\cite{Melo2017a}. As CAVs have gained momentum, different research efforts have been reported in the literature proposing coordination of CAVs. Most studies have focused on traffic bottlenecks such as merging roadways, urban intersections, and speed reduction zones \cite{Lioris2017}. 
More recently, a decentralized optimal control framework was established for coordinating online CAVs in different transportation scenarios, e.g., merging roadways, urban intersections, speed reduction zones, and roundabouts. The analytical solution using a double integrator model, without considering state and control constraints, was presented in \cite{Rios-Torres2015}, \cite{Rios-Torres2}, and \cite{Ntousakis:2016aa} for coordinating online CAVs at highway on-ramps, in \cite{Malikopoulos2018a} at roundabouts, and in \cite{Zhang2016a, Zhang:2017aa} at intersections.
The solution of the unconstrained problem was also validated experimentally at the University of Delaware's Scaled Smart City using 10 CAV robotic cars \cite{Malikopoulos2018b} in a merging roadway scenario and in a corridor \cite{beaver2019demonstration}. The solution of the optimal control problem considering state and control constraints was presented in \cite{Malikopoulos2017} at an urban intersection, without considering rear-end collision avoidance constraint though. The conditions under which the rear-end collision avoidance constraint never becomes active were discussed in \cite{Malikopoulos2018c}. The potential benefits of optimally coordinating CAVs in a corridor was presented in \cite{Zhao2018} while the implications of different penetration rates of CAVs can be found in \cite{Rios2018}. 

Other efforts have used scheduling theory for addressing this problem \cite{DeCampos2015a, Colombo2015,Ahn2014,Ahn2016,DeCampos2015a}. Colombo and Del Vecchio \cite{Colombo2015} proposed a least restrictive supervisor for intersection with collision avoidance constraints. Ahn  \cite{Ahn2014} proposed the design of a supervisory controller by imposing a hard constraint on safety and studied its behavior in the presence of manual-driven cars. In a sequel paper \cite{Ahn2016}, a supervisory controller was designed for an intersection with the ability to override human-driven control input in case of a future collision. Job-shop scheduling was used to solve this problem without considering rear-end collision avoidance constraint. 

In this paper, we propose a decentralized optimal control framework of CAVs for two interconnected intersections  using scheduling theory. We use a drone to act as a \textit{coordinator} that can broadcast information with the CAVs. The proposed approach is different from other approaches  in the literature in two aspects.
\noindent First,  most papers have considered single intersections, with only a few exceptions \cite{Zhang2016a, mahbub2019energy}, where the solution included two isolated coordinators and did not consider right and left turns.  In our approach, we consider only one coordinator (the drone) and include right and left turns. Second, the majority of the papers in this area have used centralized scheduling which has some limitations in high traffic flow. In this paper, we use a decentralized scheduling approach that can make the system more robust in case one agent breaks down.

The remainder of the paper proceeds as follows. 
In Section \Rmnum{2}, we introduce the modeling framework, and the formulation of the minimization time-optimal and scheduling problems. In Section \Rmnum{3}, we present the analytical solution of the optimal control problem. Finally, we provide simulation results in Section \Rmnum{4}, and concluding remarks in Section \Rmnum{5}.

\section{Problem Formulation}

\eat{This section consists of five parts, in the first part we summarize all notations which are being used in this paper, and in the second part, the model and the assumptions are being discussed. Third part formulation for the time-optimal problem will be provided, and the forth part is devoted to scheduling and necessary formulation. The last part is describing equations for solving energy-optimal problem.}

We consider two interconnected intersections shown in Fig. \ref{fig:1}. A drone acting as a coordinator, stores information about the geometric parameters of the intersections, the paths of the CAVs crossing the intersections and schedules of CAVs. The intersection has a \textit{control zone} (Fig. \ref{fig:1}) inside of which the drone can communicate with the CAVs. The drone does not make any decision and only acts as a coordinator between CAVs. The area where lateral collision may occur is defined as a "\textit{merging zone}."

\subsection{Modeling Framework}
Let $N(t)\in\mathbb{N}$ be the number of CAVs inside the control zone at the time $t\in\mathbb{R}^{+}$ and $\mathcal{N}(t)=\{1,\ldots,N(t)\}$ be a queue that designates the order that each CAV exiting the control zone. If two or more CAVs enter the control zone at the same time, the CAV with shorter path receives lower order in the queue; however, if the length of their path is equal, then their order is chose arbitrarily.

We partition the roads around the intersections into 16 zones (Fig. \ref{fig:2}) that belong to the set $\mathcal{M}$. The set $\mathcal{M}$ has two subset, $\mathcal{M}_1$ and $\mathcal{M}_2$. The subset $\mathcal{M}_1$ includes every zone except merging zones, and $\mathcal{M}_2$ includes the merging zones. Let $\mathcal{I}_i$ denotes the path of each CAV. When CAV $i \in\mathcal{N}(t)$ enters the control zone, it creates a tuple of the zones $\mathcal{I}_i:=[m_1,\ldots,m_n]$, $m_n\in\mathcal{M}$, $n\leq 16$, that will cross.

\begin{figure}[b]
\includegraphics[width=8.5cm]{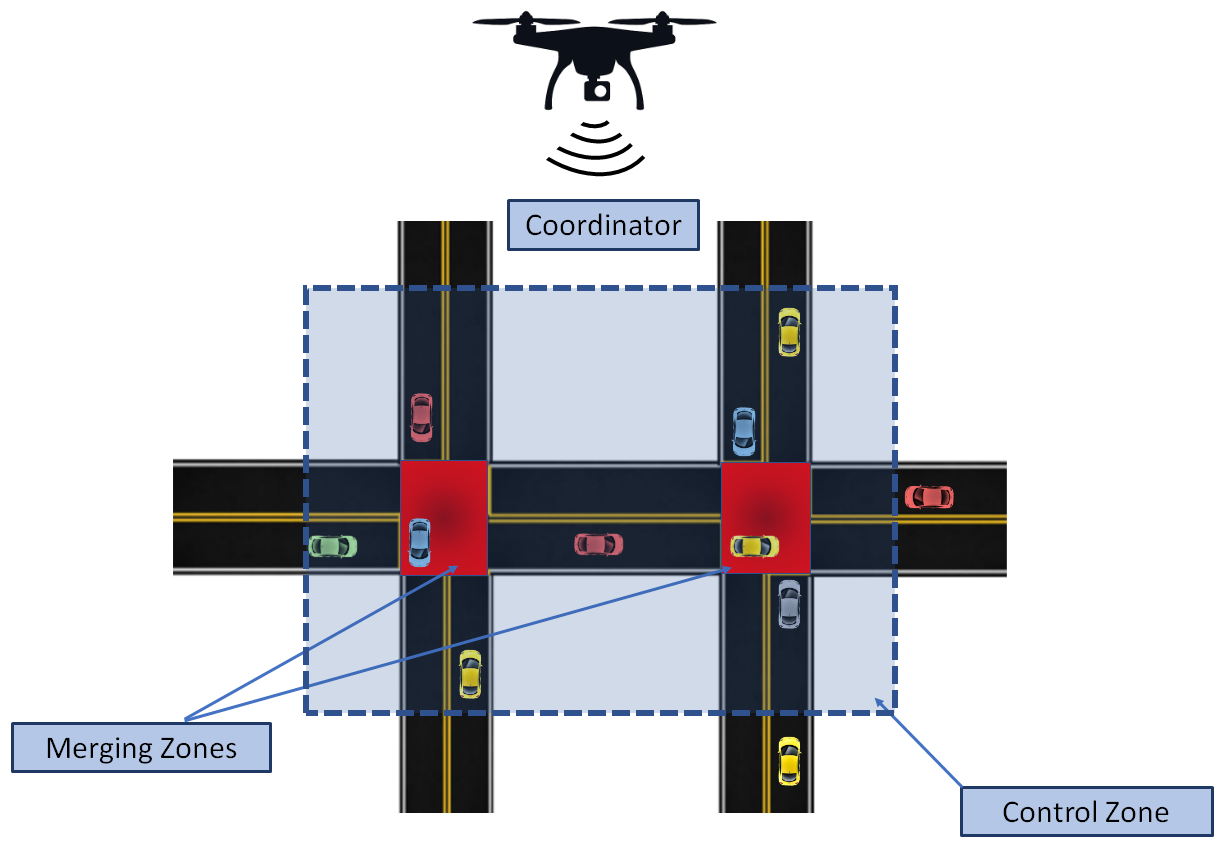}
\caption{Two interconnected intersections with a drone as a coordinator }
\label{fig:1}%
\end{figure}

\begin{definition}
For each CAV $i\in\mathcal{N}(t)$, we define a tuple $\mathcal{C}_{i,j}$ of conflict zones with CAV $j\in\mathcal{N}(t)$, $j<i$ :

\begin{gather}\label{1a}
\mathcal{C}_{i,j}=[m \quad|\quad m\in\:\mathcal{I}_i \: \wedge \:m\in\:\mathcal{I}_j ],
\end{gather}
where the symbol ``$\wedge$'' corresponds to the logical ``AND.''
\end{definition}
In Definition 1, the condition $j<i$ implies that CAV $i$ only considers CAVs that already within the control zone, i.g., their position in the queue is lower than $i$. Therefore, CAVs that are already inside the control zone do not need address potential conflicts with a new CAV that enters the control zone.

\eat{
\begin{equation}\label{1a}
\mathcal{C}_{i,j}=[\textbf{c}_{i,j}^m \quad|\quad m\in\:\mathcal{I}_i \: \wedge \:m\in\:\mathcal{I}_j ],
\end{equation}
with 
\begin{equation}\label{2a}
\textbf{c}_{i,j}^m=\left[b ~ d_s ~ d_e\right]^{T},
\end{equation}
where $b\in\{0,1\}$ indicates the type of the conflict zone and $b=1$ if lateral collision may occur, while $b=0$ if only rear-end collision may occur, $d_s\in \mathbb{R}^+$ is the distance from the entry point of the control zone to the entry point of the conflict zone $\textbf{c}_{i,j}^m$, and $d_e\in \mathbb{R}^+$ is the distance from the entry point of the control zone to the exit point of the conflict zone $\textbf{c}_{i,j}^m$. Finally, the symbol ``$\wedge$'' corresponds to the logical ``AND.''
\end{definition}
In the Definition 1, the condition $j<i$ implies that conflict tuple of CAV $i$ only includes CAVs with lower index than $i$ from the queue $\mathcal{N}(t)$. CAV $i$ does not take any CAV $k>i$ with higher index into account, since CAV $k$ will consider CAV $i$.   
}

\begin{figure}
\includegraphics[width=8.5cm]{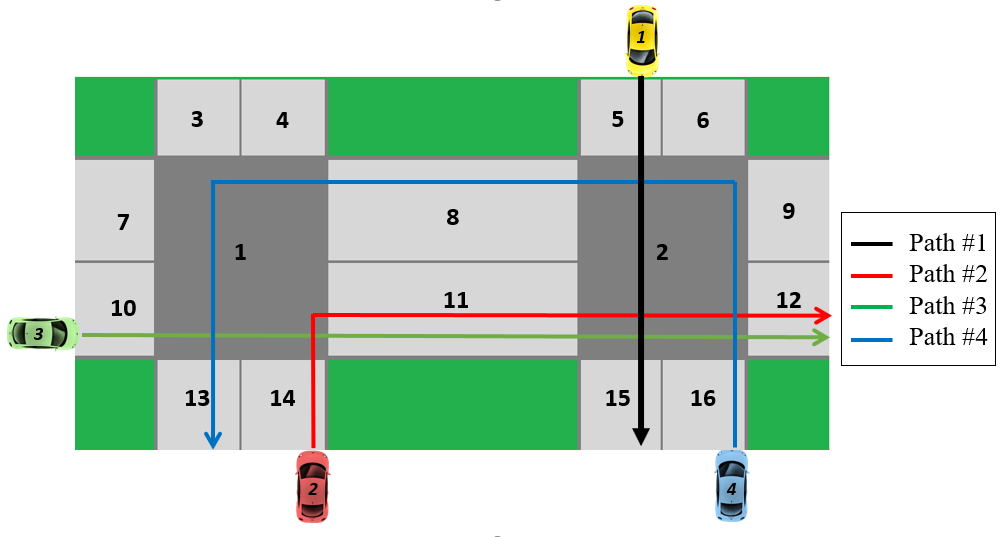}
\caption{Zones numbered topologically and the fixed path for each CAV is shown.}
\label{fig:2}%
\end{figure}
\eat{
For example in Fig. \ref{fig:2} CAVs $\#1,\#2,\#3$ and $\#4$ have the following conflict sets:
\begin{gather}
\mathcal{C}_{3,2}=\{\textbf{c}_{2,1}^1,\textbf{c}_{2,1}^{11},\textbf{c}_{2,1}^2,\textbf{c}_{2,1}^{12}\},
\end{gather}
\begin{gather}
\mathcal{C}_{3,1}=\{\textbf{c}_{3,1}^2, \textbf{c}_{3,1}^1 \},
\end{gather}
\begin{gather}
\mathcal{C}_{3,2}=\{\textbf{c}_{3,2}^2,\textbf{c}_{3,2}^1 \}.
\end{gather}
}

\subsection{Vehicle model and assumptions}
We represent the dynamics of CAV $i\in\mathcal{N}(t)$ with a state equation,
\begin{equation}
\dot{x}_i(t)=f(t,x_{i}(t),u_{i}(t)),\qquad x_{i}(t_{i}^{0})=x_{i}^{0},\label{eq:model}\\
\end{equation}
where $t\in\mathbb{R}^{+}$, and $x_{i}(t)=\left[p_{i}(t) , v_{i}(t)\right]  ^{T}$, $u_{i}(t)$ are the state and control input of CAV $i$ at time $t$.
Let $t_{i}^{0}$ be the time that CAV $i\in\mathcal{N}(t)$
enters the control zone and $x_{i}^{0}=\left[p_{i}^{0} , v_{i}^{0} \right] ^{T}$ be the state at this time. Let $t_{i}^{f}$ be the time that CAV $i\in\mathcal{N}(t)$ exits the control zone.

We model each CAV  $i\in\mathcal{N}(t)$ as a double integrator 
\begin{gather}\label{27a}
\dot{p}_i=v_i(t),\\
\dot{v}_i=u_i(t),\nonumber
\end{gather}
where $p_{i}(t)\in\mathcal{P}_{i}$, $v_{i}(t)\in\mathcal{V}_{i}$, and
$u_{i}(t)\in\mathcal{U}_{i}$ denote position, speed and acceleration. 
For each CAV $i\in\mathcal{N}(t)$ acceleration and speed is bounded with the following constraints:
\begin{gather}\label{uconstraint}
    u_{i,min}\leq u_i(t)\leq u_{i,max},
\end{gather}
\begin{gather}\label{vconstraint}
    0 < v_{i,min}\leq v_i(t)\leq v_{i,max},
\end{gather}
where $u_{i,min},u_{i,max}$ are the minimum and maximum control inputs for each CAV $i\in\mathcal{N}(t)$ and $v_{min},v_{max}$ are the minimum and maximum speed limit respectively. For simplicity, we do not consider CAV diversity. Thus, in the rest of the paper, we set $u_{i,min}=u_{min}$ and $u_{i,max}=u_{max}.$
The sets $\mathcal{P}_{i}$,
$\mathcal{V}_{i}$ and $\mathcal{U}_{i}$, $i\in\mathcal{N}(t),$
are complete and totally bounded subsets of $\mathbb{R}$.

\begin{assumption}
The path of each CAV is predefined and all CAVs in the control zone have access to this information through the drone.
\end{assumption}
\begin{assumption}
CAVs travel inside the merging zones with a constant speed which is known a priori.
\end{assumption}

\begin{assumption}
There is no error or delay in any communication between CAV to CAV, and CAV to drone.
\end{assumption}
\eat{
\begin{assumption}
For each CAV $i\in \mathcal{N}(t)$, the position $p_i(t)$ is measured as a distance relative to the CAV's position at the entrance to the control zone.
\end{assumption}}

The first assumption ensures that when a new CAV enters the control zone, it has access to the memory of the drone and discerns the path of CAVs that are inside the control zone.  The second assumption is to enhance safety awareness inside the merging zones. Assumption 2 aim each CAV to solve a scheduling problem upon arriving the control zone. The third assumption may be strong but it is relatively straightforward to relax it, as long as the noise or delays are bounded. 
    
    \eat{
    The fourth assumption helps us address the problem in one dimension, which is much easier.}
    
Next, we describe three optimization problems: (1) a time-minimization problem, (2) a scheduling problem, and (3) an energy-minimization problem. 
Every time a CAV $i$ enters the control zone, it solves a time minimization problem and a scheduling problem. With time-minimization problem, each CAV seeks to minimize travel time in each zone. Then, the results of the time-minimization problem are used to solve the scheduling problem for each zone along their path.  Since the scheduling problem considers safety, a solution may not exists. In this case, the CAV solves an energy-minimization problem to derive its optimal acceleration/deceleration from the time it enters until the time it exits the control zone.
\subsection{Time Optimization Problem}
We seek to develop a framework that minimizes the travel time of CAVs in the interconnected intersections using scheduling and optimal control theory. 
We formulate the following minimization problem:
\begin{problem}\label{problem1}
\begin{equation}\label{2aaq}
\begin{array}{ll}
\min\limits_{\textit{u}_i\in\mathcal{U}_i} \quad J_1(p_i(.),v_i(.),u_i(.),t)=t^{e,m}_i-t^{s,m}_i, \\
\emph{subject to:}\quad(\ref{27a}), (\ref{uconstraint}), (\ref{vconstraint}),\\
\\
\emph{and given }\quad p^{s,m}_i, v^{s,m}_i, p^{e,m}_i, v^{e,m}_i. \\
\end{array}  
\end{equation}
\end{problem}{}
In Problem \ref{problem1}, $[ p^{s,m}_i,v^{s,m}_i]^T$ and $[p^{e,m}_i,v^{e,m}_i]^T$ are the initial and final states of CAV $i\in\mathcal{N}(t)$ at zone $m$, where $m\in\mathcal{I}_i$ $\wedge$ $m\in\mathcal{M}_1$. The time that CAV $i\in\mathcal{N}(t)$ exits the zone $m\in\mathcal{I}_i$ is denoted by $t^{e,m}_i$.
For zone $m$, $m\in\mathcal{I}_i$ $\wedge$ $m\in\mathcal{M}_1$, the solution of Problem 1 yields the minimum time that CAV $i$ can travel through zone $m$.
It has been shown \cite{ross2015primer} that the optimal solution of \eqref{2aaq} is bang-bang control, e.g., full-on forward ($u=u_{max}$) followed by full-on reverse ($u=u_{min}$). For the case that $m\in\mathcal{I}_i$ $\wedge$ $m\in\mathcal{M}_2$, since speed is constant, we do not need to minimize the time. 
\eat{\begin{remark}
Process time for segment 1 and 2 is equal to 
\begin{equation}\label{2wq}
P^{i}_m=\frac{s^i_m}{v_z}\quad m\in \{1,2\}
\end{equation}
$s^i_m$ is the length of the zone $m$ for CAV $i$ and $v_z$ is the constant speed inside the merging zones $m$.
\end{remark}}
\subsection{Scheduling Problem}
Scheduling is a decision-making process which addresses the optimal allocation of resources to tasks over given time periods \cite{pinedo2016scheduling}.
The problem of coordinating CAVs at the intersection is a typical scheduling problem \cite{Colombo,Ahn2014,Ahn2016}. Thus, in what follows we use scheduling theory to find the time that CAV $i\in\mathcal{N}(t)$ has to reach the zone $m\in\mathcal{I}_i$.
Each zone $m\in\mathcal{M}$ represents a "resource," and CAVs crossing this zone are the "jobs" assigned to the resource. 
\begin{definition}
The time that a CAV $i\in\mathcal{N}(t)$  enters each zone $m\in\mathcal{I}_i$ is called "schedule" and is denoted by $T_{i}^m\in\mathbb{R}^+$.
For CAV $i\in\mathcal{N}(t)$ we define a tuple of its schedules as follows: 
\begin{equation}\label{2a1}
\mathcal{T}_{i}=[T_{i}^m~|~m\in\mathcal{I}_i ].
\end{equation}
\end{definition}{}
For each zone $m\in\mathcal{I}_i$, $i\in\mathcal{N}(t),$ the schedule $T_i^m\in\mathbb{R}^+$ is bounded by
\begin{equation}\label{Schedulecons1}
      R_i^m\leq T_i^m\leq D_i^m,\\
\end{equation}
where $R_i^m\in\mathbb{R}^+$ is the earliest feasible time that CAV $i\in\mathcal{N}(t)$ can reach the entry point of zone $m\in\mathcal{I}_i$, while $D_i^m\in\mathbb{R}^+$ is the latest feasible time. Moreover, $R_{i}^m\in\mathbb{R}^+$ and $D_{i}^m\in\mathbb{R}^+$ are called the "\textit{release time}" and the "\textit{deadline}" of the job respectively.

\eat{
For example in Fig. \ref{fig:2} for CAV $\#1$ is defined:

\begin{equation}\label{2a121}
\mathcal{T}^{1}=[{T}_{1}^{14},{T}_{1}^1,{T}_{1}^{11},{T}_{1}^2,{T}_{1}^{12}].
\end{equation}}

\begin{definition}
For each zone $m\in\mathcal{C}_{i,j}$, $i, j\in\mathcal{N}(t)$, $j<i$, the safety constraint can be restated as 
\begin{equation}\label{Schedulecons2}
  |T_i^m-T_j^m|\geq h,
\end{equation}
where $h\in\mathbb{R}^+$ is the safety time headway.
\end{definition}
\begin{problem}\label{problem2}
For each CAV $i\in\mathcal{N}(t)$ and each zone $m\in\mathcal{I}_i$ with safety time headway $h\in\mathbb{R}^+$ the scheduling problem is formulated as follows
\begin{equation}\label{Scheduling}
\begin{array}{ll}

\min\limits_{T_i^m\in\mathbb{R}^+} \quad J_2(T_i^m)=T_i^m,\\
\emph{subject to:}\quad (\ref{Schedulecons1}), (\ref{Schedulecons2}).\\
\end{array}  
\end{equation}
\end{problem}

\begin{definition}\label{processtime}
The shortest feasible time that it takes for CAV $i\in\mathcal{N}(t)$ to travel through the zone $m\in\mathcal{I}_i$ is defined as the process time and is denoted by $P_{i}^m\in\mathbb{R}^+$.
\end{definition}
 The process time is the outcome of Problem \ref{problem1}, hence ~$P_{i}^m=~t_{i}^{e,m}-t_{i}^{s,m}$.

\begin{remark}
In Problem \ref{problem2}, for each CAV $i\in\mathcal{N}(t)$ and $m\in\mathcal{I}_i$, the release time for zone $m+1\in\mathcal{I}_i$ can be computed by the process time and schedule of zone $m\in\mathcal{I}_i$ 
\begin{equation}\label{remarkproblem2}
 R_i^{m+1}=P_i^m+T_i^m,
\end{equation}
where $m+1\in\mathcal{I}_i$ is the next zone that CAV $i\in\mathcal{N}(t)$ will cross after zone $m\in\mathcal{I}_i$.
\end{remark}
\begin{proposition}\label{remark2feas}
Let $T_i^{m+1}$ be the entry time of CAV $i\in\mathcal{N}(t)$ at the zone $m+1\in\mathcal{I}_i$ and $R_i^{m+1}$ be the earliest time that CAV $i$ can reach zone $m+1$. Then the feasible time-optimal solution of CAV $i$ at the zone $m\in\mathcal{I}_i$ exist (it does not violate safety constraint at entry of zone $m+1$), if and only if $T_i^{m+1}=R_i^{m+1}$.

\end{proposition}
\begin{proof}
First we prove necessary condition. Based on Definition \ref{processtime}, the solution of time-optimal problem yields $P_{i}^m$. Substituting $P_{i}^m$ in (\ref{remarkproblem2}) gives release time for zone $m+1$. If $R_i^{m+1}$ does not violate any safety constraint at entry of zone $m+1$, it is the solution of Problem \ref{problem2} for zone $m+1$ which is $T_i^{m+1}=R_i^{m+1}$.
To prove sufficiency, we know from (\ref{Schedulecons1}), $T_i^{m+1}\geq R_i^{m+1}$. If $T_i^{m+1}=R_i^{m+1}$, this implies that CAV $i$ enters the zone $m+1$ at the earliest feasible time without violating the safety constraint. From (\ref{remarkproblem2}), CAV $i$  travels with the process time $P_i^m$ at the prior zone $m$ with  schedule $T_i^m$, which is the solution of the Problem \ref{problem1} for zone $m$. Hence, a time-optimal solution exists which does not violate safety constraint at entry of zone $m+1$.
\end{proof}

\subsection{Energy Minimization Formulation}
If at zone $m\in\mathcal{I}_i$, CAV $i\in\mathcal{N}(t)$ does not have a feasible time-optimal solution, then we impose $i$ to solve an energy minimization problem. In this problem, the cost function is the $L^2$-norm of the control input. The implications of the solution of the energy-optimal control is the minimization of transient engine operation, which has direct benefits both in fuel consumption and emissions \cite{malikopoulos2010online}. 
\begin{problem}\label{problem3}
The energy minimization problem is formulated as follows.
\begin{equation}\label{2aaq2}
 \begin{array}{ll}
 \min\limits_{\textit{u}_i\in\mathcal{U}_i} \quad J_3(p_i(.),v_i(.),u_i(.),T_i^{m},T^{m^\prime}_i)= {\frac{1}{2}} \int_{T_i^{m}}^{T^{m^\prime}_i} u_{i}(t)^2dt, \\
        
\qquad\quad\emph{subject to:}\quad(\ref{27a}), (\ref{uconstraint}), (\ref{vconstraint}),\\
\\
\emph{and given } p^{s,m}_i, v^{s,m}_i, T_i^{m}, p^{e,m}_i, v^{e,m}_i, T^{m^\prime}_i, \\
           \end{array} 
\end{equation}\end{problem}
\noindent where $[p^{s,m}_i,v^{s,m}_i]^T$ and $[p^{e,m}_i,v^{e,m}_i]^T$ are initial and final states of CAV $i\in\mathcal{N}(t)$ at zone $m\in\mathcal{I}_i$ respectively. The entry and exit time of CAV $i\in\mathcal{N}(t)$ from zone $m\in\mathcal{I}_i$ is denoted by $T_i^{m}$ and $T^{m^\prime}_i$ respectively.
\\
\begin{remark}
The exit time of CAV $i\in\mathcal{N}(t)$ from zone $m\in\mathcal{I}_i$ denoted by $T^{m^\prime}_i $ and is equal to the entry time to the zone ${m+1}\in\mathcal{I}_i$ which is the next zone that CAV will cross after zone $m$. 
\begin{equation}
    T^{m^\prime}_i=T^{m+1}_i
\end{equation}
\end{remark}
\vspace{2mm}
\section{Analytical Solution of the Problems}
In this section we provide the analytical solutions of Problem \ref{problem1} and Problem \ref{problem3}. CAVs solve the scheduling problem with mixed integer linear program. Since the closed form analytical solution of the time-minimization problem is available, the mixed integer linear program can be computed as CAVs enter the control zone.
 For the sake of simplicity in notation, in the following section we use $p^{s}_i,v^{s}_i,p^{e}_i,v^{e}_i$ instead of $p^{s,m}_i,v^{s,m}_i,p^{e,m}_i,v^{e,m}_i.$
\subsection{Solution of the time minimization problem }

The solution of Problem \ref{problem1}, using Hamiltonian analysis, includes solving a system of non-linear, non-smooth equations \cite{ross2015primer}. Pontryagin\cite{pontryagin2018mathematical} solved the problem in a simpler way using a graphical approach. Since the problem's state space is two dimensional, Pontryagin noted that the optimal controls are either $u=u_{max}$, or $u=u_{min}$, which can be characterized as parabolas in two-dimensional state space. 
\eat{
If the problem has a solution, there is an intermediate point that the CAV should travel from either with $u=u_{max}$, or $u=u_{min}$ to reach the final state.
}
Based on the initial state, it should traverse on one of the parabolas to reach the intermediate point then traverse from the intermediate point to reach the target point. The aforementioned intermediate point is also called a switching point, since the control input will be changed at this point. 

\begin{proposition}\label{proposition2}
Let $v^e_i$ be the speed of CAV $i\in\mathcal{N}(t)$ at the end of the zone $m\in\mathcal{I}_i$. Let $u_{max}$ and $u_{min}$ be the maximum acceleration and deceleration respectively. If CAV $i$ cruises from the initial state $x_{i}^{s}=[p^s_i,v^s_i]^T$ to the final position $p^e_i$, then the speed of CAV $i\in\mathcal{N}(t)$ at the end of the zone $m\in\mathcal{I}_i$ should be bounded as follows:
\begin{gather}\label{vforproposition2}
\sqrt[]{2u_{min}\cdot(p^e_i-p^s_i)+{v^s_i}^2}\leq v^e_i \leq \sqrt[]{2u_{max}\cdot(p^e_i-p^s_i)+{v^s_i}^2}
\end{gather}
\end{proposition}

\begin{proof}
If CAV $i\in\mathcal{N}(t)$ accelerates with $u_{max}$, then from (\ref{27a}) 
\begin{gather}
{v^{e}_i}^2-{v^s_i}^2=2u_{max}\cdot(p^e_i-p^s_i),    
\end{gather}
or
\begin{gather}\label{finalvel}
 v^{e}_i=\sqrt[]{2u_{max}\cdot (p^e_i-p^s_i)+{v^s_i}^2}.
\end{gather}
Similarly, if CAV $i\in\mathcal{N}(t)$ decelerates with $u_{min}$, its final speed is
\begin{gather}\label{finalvel2}
 v^{e}_i=\sqrt[]{2u_{min}\cdot(p^e_i-p^s_i)+{v^s_i}^2}.
\end{gather}
The maximum and minimum final speeds are computed in In \eqref{finalvel} and \eqref{finalvel2} respectively. Therefore, the final speed is bounded from (\ref{finalvel}) and (\ref{finalvel2}).
\end{proof}
\eat{
The result of the Proposition \ref{proposition2} is also shown in Fig. \ref{fig:4}
\begin{figure}
\includegraphics[width=8.5cm]{proposition2.png}
\caption{Feasible region (filled with green color) for the velocity of CAV $i\in\mathcal{N}(t)$ at the end of the zone $m\in\mathcal{I}_i$.}
\label{fig:4}%
\end{figure}}
\begin{figure}
	\centering
	\includegraphics[width=6cm]{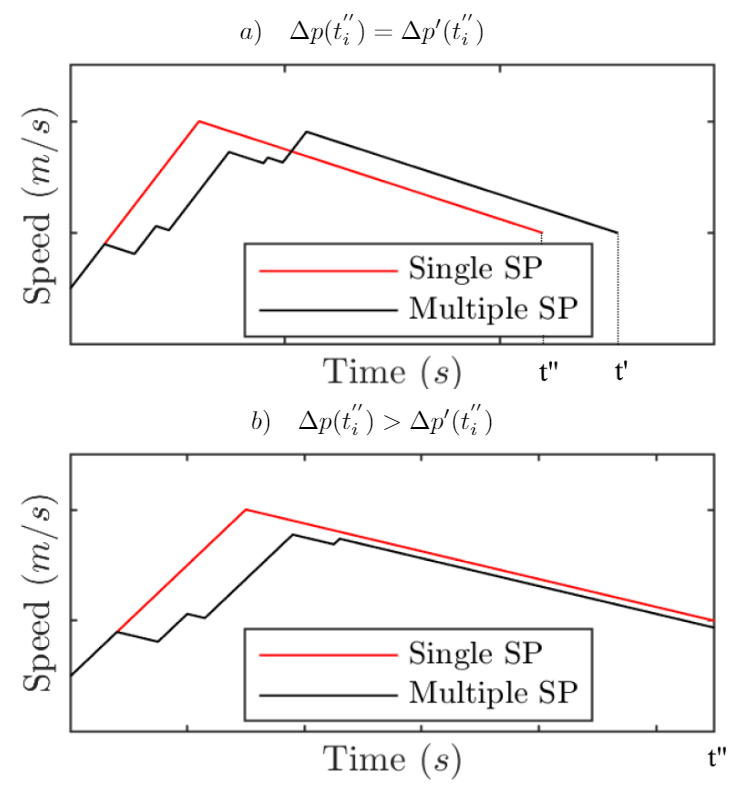}
	\caption{Area under speed-time plot for multiple switching points and a single switching point. }
	\label{extrafig:1}%
\end{figure}
\begin{lemma}\label{Switching point}
If CAV $i\in\mathcal{N}(t)$ has a feasible time-optimal solution at zone $m\in\mathcal{I}_i$ (Proposition \ref{remark2feas}), it has at most one switching point.
\end{lemma}
\begin{proof}
We have three cases to consider:

Case 1: In \eqref{vforproposition2}, if  $v^{e}_i$ is equal to the lower bound, CAV $i\in\mathcal{N}(t)$ arrives at the end of zone $m\in\mathcal{I}_i$ with maximum deceleration $u_{min}$. In this case, there is no switching point. 

Case 2: In \eqref{vforproposition2}, if  $v^{e}_i$ is equal to the upper bound, CAV $i\in\mathcal{N}(t)$ arrives at the end of zone $m\in\mathcal{I}_i$ with maximum acceleration $u_{max}$. In this case, there is no switching point. 

Case 3: In this case, from \eqref{vforproposition2} the CAV has at least one switching point.
Let $t^{\ast}_i\in\mathbb{R}^+$ be the minimum travel time. Let $t^{\prime}_i\in\mathbb{R}^+$ and $t^{''}_i\in\mathbb{R}^+$ be the time that it takes for CAV $i$ to travel at zone $m$ with more than one switching points, and exactly one switching point respectively. Let $\Delta p^\prime(t)$ and $\Delta p(t)$ denote the distance traveled by CAV $i$ at zone $m$ at time $t$ for more than one switching point and one switching point respectively, which can be represented by the area under $v-t$ diagram. 
First consider the case that the area under $v-t$ diagram for one switching point is greater than multiple switching points in the interval $t\in[0,t^{''}_i]$, then we can write
\begin{gather}\label{Singlepoint1}
\quad  \Delta p(t^{''}_i)>\Delta p^\prime(t^{''}_i).
\end{gather}
The left hand side of (\ref{Singlepoint1}) is equal to $p_i^e-p_i^s$,
\begin{gather}\label{Singlepoint2}
   p_i^e-p_i^s>\Delta p^\prime(t^{''}_i).
\end{gather}
We know that 
\begin{gather}\label{Singlepoint3}
   p_i^e-p_i^s=\Delta p^\prime(t^{\prime}_i).
\end{gather}
Hence from (\ref{Singlepoint3}) and (\ref{Singlepoint2}), we have \begin{gather}\label{Singlepoint4}
 \Delta p^\prime(t^{\prime}_i)>\Delta p^\prime(t^{''}_i).
\end{gather}
Since $\Delta p^\prime(t)$ is the distance travelled by CAV $i$, it is an increasing function. Therefore, from (\ref{Singlepoint4}) 
\begin{gather}\label{Singlepoint5}
t^{\prime}_i>t^{''}_i.
\end{gather}
Hence the minimum travel time is:
\begin{gather}\label{Singlepoint6}
t^{\ast}_i=t^{''}_i.
\end{gather}
For the case $\Delta p(t^{''}_i)=\Delta p^\prime(t^{''}_i)$,  it can be shown that for satisfying the endpoint speed condition, $t^\prime$ has to be greater than $t^{''}$.
\end{proof}
An example of multiple switching point and a single switching point is shown in Fig. \ref{extrafig:1}.

\begin{theorem}\label{lem1}
Let $x_{i}^{s}=[p^s_i,v^s_i]^T$ and $x_{i}^{e}=[p^e_i,v^e_i]^T$ be the initial and final state of CAV $i\in\mathcal{N}(t)$.
If CAV $i$ has a feasible time-optimal solution at zone $m\in\mathcal{I}_i$ (Proposition \ref{remark2feas}), then it has to accelerate with $u=u_{max}$ and then decelerate with $u=u_{min}$.
\end{theorem}
\begin{proof}
There are two cases to consider:

Case 1: CAV $i$ decelerates with $u_{min}$ and then accelerates with $u_{max}$.
The intermediate state $x_{i}^{c^\prime}=[p^{c^\prime}_i,v^{c^\prime}_i]^T$ for case 1 is found with using the dynamics of the model (\ref{27a}), and solving the system of equations below:
\begin{gather}\label{cprinesystem1}
 \left\{ \begin{array}{ll}
{v^{c^\prime}_i}^2-{v^s_i}^2=2u_{min}\cdot(p^{c^\prime}_i-p^s_i)\\
{v^e_i}^2-{v^{c^\prime}_i}^2=2u_{max}\cdot(p^e_i-p^{c^\prime}_i)
           \end{array} \right..
\end{gather}
To simplify notation, we denote $u_{min}=-a$ and $u_{max}=a$, hence
\begin{gather}\label{cprinesystem2}
 \left\{ \begin{array}{ll}
{v^{c^\prime}_i}^2-{v^s_i}^2=-2a(p^{c^\prime}_i-p^s_i)\\
{v^e_i}^2-{v^{c^\prime}_i}^2=2a(p^e_i-p^{c^\prime}_i)
           \end{array} \right..
\end{gather}
Solving (\ref{cprinesystem2}) yields
\begin{gather}\label{cprinesystem2sol1}
p^{c^\prime}_i=\frac{-({v^e_i}^2-{v^s_i}^2)+2a(p^s_i+p^e_i)}{4a}.
\end{gather}

\begin{gather}\label{cprinesystem2sol2}
v^{c^\prime}_i=\sqrt{{v^s_i}^2-2a(p_i^{c^\prime}-p_i^s)}.
\end{gather}
Plugging (\ref{cprinesystem2sol1}) in (\ref{cprinesystem2sol2}) we have
\begin{gather}\label{cprinesystem2sol3}
v^{c^\prime}_i=\sqrt{\frac{{v^s_i}^2+{v^e_i}^2-2a(p^e_i-p^s_i)}{2}}.
\end{gather}
The total travel time at zone $m$ is

\begin{gather}\label{cprinetime}
t^{e^\prime}_i=\frac{v^{c^\prime}_i -v^s_i}{-a}+\frac{v^e_i-v^{c^\prime}_i}{a},
\end{gather}
or in a simplified form
\begin{gather}\label{cprinetime2}
t^{e^\prime}_i=\frac{v^s_i+v^e_i-2v^{c^\prime}_i}{a}.
\end{gather}
Case 2: CAV $i$ accelerates first and then decelerates.
The intermediate state is
\eat{
The intermediate state $x_{i}^{c}=[p^{c}_i,v^{c}_i]^T$ is found with the using dynamics of the model (\ref{27a}) and solving the system of equations below:
\begin{gather}\label{csystem1}
 \left\{ \begin{array}{ll}
{v^{c}_i}^2-{v^s_i}^2=2u_{max}\cdot(p^{c}_i-p^s_i)\\
{v^e_i}^2-{v^{c}_i}^2=2u_{min}\cdot(p^e_i-p^{c}_i)
           \end{array} \right.,
\end{gather}
with the same assumption of $u_{min}=-a$ and $u_{max}=a$, it yields: 
\begin{gather}\label{csystem2}
 \left\{ \begin{array}{ll}
{v^{c}_i}^2-{v^s_i}^2=2a(p^{c}_i-p^s_i)\\
{v^e_i}^2-{v^{c}_i}^2=-2a(p^e_i-p^{c}_i)
           \end{array} \right.,
\end{gather}
solving the system of equations in (\ref{csystem2}) yields:}
\begin{gather}\label{csystem2sol1}
p^{c}_i=\frac{({v^e_i}^2-{v^s_i}^2)+2a(p^s_i+p^e_i)}{4a},
\end{gather}

\begin{gather}\label{csystem2sol2}
v^{c}_i=\sqrt{{v^s_i}^2+2a(p_i^{c}-p_i^s)}.
\end{gather}
Plugging (\ref{csystem2sol1}) in (\ref{csystem2sol2}) we have
\begin{gather}\label{csystem2sol3}
v^{c}_i=\sqrt{\frac{{v^s_i}^2+{v^e_i}^2+2a(p^e_i-p^s_i)}{2}}.
\end{gather}
The total travel time at zone $m$ is 

\begin{gather}\label{ctime}
t^{e}_i=\frac{v^{c}_i -v^s_i}{a}+\frac{v^e_i-v^{c}_i}{-a},
\end{gather}
or in a simplified form
\begin{gather}\label{ctime2}
t^{e}_i=\frac{-(v^s_i+v^e_i)+2v^{c}_i}{a}.
\end{gather}
If the switching point exists, the distance travelled  while $u_{max}$ (or $u_{min}$) is shorter than the total distance $p^e_i-p^s_i$. Let $l_i=p^e_i-p^s_i$. We have
\begingroup
\allowdisplaybreaks
\begin{align}
&|{v^e_i}^2-{v^s_i}^2|<2al_i \nonumber\\
\Longleftrightarrow & ({v^e_i}^2-{v^s_i}^2)^2<(2al_i)^2 \nonumber\\
\Longleftrightarrow &{v^e_i}^4+{v^s_i}^4-2{v^e_i}^2{v^s_i}^2<4a^2l_i^2 \nonumber\\
\Longleftrightarrow
 &{v^e_i}^4+{v^s_i}^4-4a^2l_i^2+2{v^e_i}^2{v^s_i}^2<4{v^e_i}^2{v^s_i}^2 \nonumber\\
\Longleftrightarrow &({v^e_i}^2+{v^s_i}^2)^2-(2al_i)^2<(2{v^e_i}{v^s_i})^2 \nonumber\\
\Longleftrightarrow &({v^e_i}^2+{v^s_i}^2-2al_i)({v^e_i}^2+{v^s_i}^2+2al_i)<(2{v^e_i}{v^s_i})^2 \nonumber\\
\Longleftrightarrow &\sqrt{({v^e_i}^2+{v^s_i}^2-2al_i)({v^e_i}^2+{v^s_i}^2+2al_i)}<2{v^e_i}{v^s_i} \nonumber\\
\Longleftrightarrow &2\sqrt{(\dfrac{{v^e_i}^2+{v^s_i}^2-2al_i}{2})(\dfrac{{v^e_i}^2+{v^s_i}^2+2al_i}{2})} \nonumber\\
&+{v^e_i}^2+{v^s_i}^2<{v^e_i}^2+{v^s_i}^2+2{v^e_i}{v^s_i} \nonumber\\
\Longleftrightarrow 
&2\sqrt{(\dfrac{{v^e_i}^2+{v^s_i}^2-2al_i}{2})(\dfrac{{v^e_i}^2+{v^s_i}^2+2al_i}{2})} \nonumber\\
&+(\dfrac{{v^e_i}^2+{v^s_i}^2-2al_i}{2})(\dfrac{{v^e_i}^2+{v^s_i}^2+2al_i}{2}) \nonumber\\
&<({v^e_i}+{v^s_i})^2 \nonumber\\
\Longleftrightarrow &(v^{c^\prime}_i+v^{c}_i)^2<({v^e_i}+{v^s_i})^2 \nonumber\\
\Longleftrightarrow &(v^{c^\prime}_i+v^{c}_i)<({v^e_i}+{v^s_i}) \nonumber\\
\Longleftrightarrow &2(v^{c^\prime}_i+v^{c}_i)<2({v^e_i}+{v^s_i}) \nonumber\\
\Longleftrightarrow &2v^{c}_i-({v^e_i}+{v^s_i})<-2v^{c^\prime}_i+({v^e_i}+{v^s_i}) \nonumber\\
\Longleftrightarrow
&\dfrac{2v^{c}_i-({v^e_i}+{v^s_i})}{a}<\dfrac{-2v^{c^\prime}_i+({v^e_i}+{v^s_i})}{a} \nonumber\\
\Longleftrightarrow & t^{e}_i<t^{e^\prime}_i \label{lem1proof}
\end{align}
\endgroup
Hence, the travel time for CAV $i$ at zone $m$ is shorter in case 2.

\end{proof}
\eat{
The Fig. \ref{fig:3} shows the result of the lemma \ref{lem1}
\begin{figure}
\includegraphics[width=8.5cm]{lemma1.png}
\caption{The profile of speed-time for a) acceleration-deceleration and b) deceleration-acceleration .}
\label{fig:3}%
\end{figure}}

\begin{lemma}\label{lemm2}
Let $x_{i}^{0}=[p^s_i,v^s_i]^T$ and $x_{i}^{e}=[p^e_i,v^e_i]^T$ be the initial state and final state of CAV $i\in\mathcal{N}(t)$ travelling in zone $m\in\mathcal{I}_i$. Let $x_{i}^{c}=[p^c_i,v^c_i]^T$ be the intermediate state at switching point. Then, the intermediate state is 
\begin{gather}\label{zz}
p^{c}_i=\frac{{v^e_i}^2-{v^s_i}^2+2(u_{max}p^s_i-u_{min}p^e_i)}{2(u_{max}-u_{min})},
\end{gather}

\begin{gather}\label{z}
v^{c}_i=\sqrt{{v^s_i}^2+2u_{max}\cdot(p_i^c-p_i^s)}.
\end{gather}

\end{lemma}
\vspace{1mm}
\begin{proof}
From (\ref{27a}) and Theorem \ref{lem1}, the intermediate states can be found by solving the following system of equations 
\begin{gather}\label{z2}
 \left\{ \begin{array}{ll}
{v^c_i}^2-{v^s_i}^2=2u_{max}\cdot(p^c_i-p^s_i)\\
{v^e_i}^2-{v^c_i}^2=2u_{min}\cdot(p^e_i-p^c_i)
           \end{array}, \right. 
\end{gather}
that yields (\ref{zz}) and (\ref{z}).

\end{proof}

\begin{lemma}
Let $t_i^c$, and $t^e_i$ be the time at intermediate and final state for CAV $i\in \mathcal{N}(t)$ at zone $m\in\mathcal{I}_i$ respectively. Then $t_i^c$ and $t^e_i$ are 
\begin{gather}\label{tc}
t^c_i=\frac{v^c_i -v^s_i}{u_{max}},
\end{gather}
\begin{gather}\label{te}
t^e_i=\frac{v^c_i -v^s_i}{u_{max}}+\frac{v^e_i-v^c_i}{u_{min}}.
\end{gather}
\end{lemma}

\begin{proof}
The control input of CAV $i\in\mathcal{N}(t)$ at zone $m\in\mathcal{I}_i$ is consists of two parts, acceleration with ${u_{max}}$ and deceleration with ${u_{min}}$. Since the control input is constant, the total time for each part can be found by integrating (\ref{27a})  

\begin{gather}\label{lem3proof2}
\begin{array}{ll}
v^c_i -v^s_i=u_{max}\cdot t^c_i, & \forall t\in[0,t^c_i],\\
\\
v^e_i -v^c_i=u_{min}\cdot(t^e_i-t^c_i), & \forall t\in[t^c_i,t^e_i].
\end{array}
\end{gather}
Solving (\ref{lem3proof2}) for $t^e_i$ and $t^e_i$ yields (\ref{tc}) and (\ref{te}).

\end{proof}
\begin{theorem}
For CAV $i\in\mathcal{N}(t)$ at zone $m\in\mathcal{I}_i$, the process time $P_i^m$ is
\begin{equation}\label{theorem1}
 \left\{ \begin{array}{ll}
 P_{i}^m=t_i^e\quad m\in \mathcal{M}_1\\
 \\
 P_{i}^m=\dfrac{p_i^e-p_i^s}{v_z}\quad m\in\mathcal{M}_2\
           \end{array},\right.
           \end{equation}
where $v_z$ is the constant speed inside the merging zones $m\in\mathcal{M}_2$.
\end{theorem}

\begin{proof}
 For $m\in \mathcal{M}_1$ the shortest feasible time that CAV $i\in \mathcal{N}(t)$ can travel through the zone $m$ with known initial and final state is found from the solution of the Problem \ref{problem1}, which is derived in equation \eqref{te}. 
 
 CAV $i$ in any zone $m\in \mathcal{M}_2$, is assumed to travel with the constant and imposed speed $v_z$, the shortest time that it takes to cross zone $m$ is simply found from division of length of the zone by the speed.
\end{proof}{}

\vspace{1mm}
\eat{ CAV $j\in N(t)$ solves the time minimization problem for each segment $m\in M$ in its path and then it will computed the minimum process time$P^{j}_m$ as : 
\begin{gather}
P^{j}_m=\frac{2v_i -(v^s+v^f)}{a}
\end{gather}}
\begin{remark}
If CAV $i\in\mathcal{N}(t)$ has a feasible time-optimal solution at zone $m\in\mathcal{I}_i$ (Proposition \ref{remark2feas}), the  real time feedback control is
\begin{gather}\label{e2aq2}
 u_i(t)=\left\{ \begin{array}{ll}
u_{max}& \emph{if}\quad T_i^m\leq t<T_i^m+t^c_i\\
u_{min} & \emph{if}\quad T_i^m+t^c_i\leq t\leq T_i^m+ t^e_i\\
           \end{array}. \right. 
\end{gather}
\end{remark}
Substituting (\ref{e2aq2}) in (\ref{27a}), we can find the optimal position and speed for each CAV $i\in\mathcal{N}(t)$ at zone $m\in\mathcal{I}_i$, namely 
\begin{gather}\label{vstar}
 v_i^{\ast}(t)=u_it+b_i,
\end{gather}
\begin{gather}\label{pstar}
p_i^{\ast}(t)=\frac{1}{2}u_it^2+b_it+c_i,
\end{gather}
where \(b_i,c_i\) are integration constants, which can be computed by using initial and final conditions in (\ref{2aaq}). In particularly, for $t\in[T_i^m, T_i^m+t^c_i)$, using (\ref{vstar}) with the initial condition $v_i(T_i^m)=v^s_i$, (\ref{pstar}) with the initial condition $p_i(T_i^m)=p^s_i$, a system of equations can be formulated in the form of \(\textbf{E}_i \textbf{b}_i=\textbf{q}_i\):
\begin{gather}
  \begin{bmatrix}
T_i^m & 1\\
1 & 0
   \end{bmatrix}
   \cdot
   \begin{bmatrix}
   b_i\\
   c_i
   \end{bmatrix}
   =   \begin{bmatrix}
   p^s_i-\frac{1}{2}u_{max}\cdot (T_i^m)^2\\
    v^s_i-u_{max}\cdot T_i^m
   \end{bmatrix}.
\end{gather}

However, for $t\in[T_i^m+t^c_i, T_i^m+t^e_i]$, using (\ref{vstar}) with the final condition $v_i( T_i^m+t^e_i)=v^e_i$, (\ref{pstar}) with the final condition $p_i( T_i^m+t^e_i)=p^e_i$, a system of equations will be as following:

\begin{gather}
  \begin{bmatrix}
T_i^m+t^e_i & 1\\
1 & 0
   \end{bmatrix}\cdot
   \begin{bmatrix}
   b_i\\
   c_i
   \end{bmatrix}
   =   \begin{bmatrix}
   p^e_i-\frac{1}{2}u_{min}\cdot(T_i^m+t^e_i)^2\\
    v^e_i-u_{min}\cdot (T_i^m+t^e_i)
   \end{bmatrix}.
\end{gather}

\subsection{Solution of the Energy minimization problem }

The derivation solution of the analytical, closed-form solution of Problem \ref{problem3} has been presented \cite{Malikopoulos2017,Malikopoulos2018c}.
If none of the contraints are active the optimal control input, speed, and position of each CAV $i\in\mathcal{N}(t)$ are
\begin{gather}\label{27}
 u_i^{\ast}=a_it+b_i,\\
 v_i^{\ast}=\frac{1}{2}a_it^2+b_it+c_i,\\
p_i^{\ast}=\frac{1}{6}a_it^3+\frac{1}{2}b_it^2+c_it+d_i,
\end{gather}
In above equations \(a_i,b_i,c_i,d_i\) are integration constants, which can be found by plugging initial state at zone $m\in\mathcal{I}_i$ $p_i(T_i^{m})=p^s_i, v_i(T_i^{m})=v^s_i$ and final states $p_i(T_i^{m^\prime})=p^e_i, v_i(T_i^{m^\prime})=v^e_i$ in equations.
Thus, a system of equations can be formed in form of \(\textbf{T}_i \textbf{b}_i=\textbf{q}_i\)

\begin{gather}\label{energy-optimal-matrix}
  \begin{bmatrix}
   \frac{1}{6}(T_i^{m})^3&
   \frac{1}{2}(T_i^{m})^2&
   (T_i^{m})&
   1\\
   \\
   \frac{1}{2}(T_i^{m})^2&
   (T_i^{m})&
   1&
   0\\
   \\
   \dfrac{1}{6}(T_i^{m^\prime})^3&
   \dfrac{1}{2}(T_i^{m^\prime})^2&
   (T_i^{m^\prime})&
   1\\
   \\
   \frac{1}{2}(T_i^{m^\prime})^2&
   (T_i^{m^\prime})&
   1&
   0\\
   \end{bmatrix}
   \cdot
     \begin{bmatrix}
   a_i\\
   b_i\\
   c_i\\
   d_i
   \end{bmatrix}
   =
     \begin{bmatrix}
   p_i^s\\
   v_i^s\\
   p^e_i\\
   v^e_i
   \end{bmatrix}
\end{gather}
Note that since (\ref{energy-optimal-matrix}) can be computed
online, the controller may re-evaluate the four constants at any time $t\in[T_i^{m},T_i^{m^\prime}]$ and update (\ref{27}) accordingly.
\section{Simulation}
To evaluate the effectiveness of the proposed framework, we considered the coordination of $16$ CAVs in simulation. We considered two symmetrical intersections, where the length of each road connecting to the intersections is $400$ $m$, and the length of the merging zones are $30$ $m$. The minimum time headway is $1.5$~$s$.
The maximum acceleration limit is set to $3$ $m/s^2$ and the minimum deceleration limit is $-3$ $m/s^2$. The speed at merging zones is set $15$ $m/s$. 
The CAVs enter the control zones from four different paths in order (shown in Fig. \ref{fig:2}) at random time with uniform distribution between $0$ and $20$ $s$. Initial time were checked to not violate the safety constrained (\ref{Schedulecons2}) at entrance.

The speed profile for the first $4$ CAVs is shown in  Fig. \ref{fig:5}. CAVs $1-4$ enter the control zone at random time, from path $1$ to path $4$ respectively. It can be seen in Fig. \ref{fig:5} that CAV $1$, CAV $2$ and  CAV $4$ are following time-optimal trajectory on each zone of their path. They accelerate  in each zone $m\in\mathcal{M}_1$  until certain time (switching point) and then decelerate to reach the next zone.
However, CAV $3$ has a quadratic form at the beginning of its path. The quadratic form of velocity profile results from solving the energy-minimization problem implying that the time-optimal solution was not feasible for its first zone based on the schedule tuple (see the Table \ref{tab:experimentalResults}). Therefore, acceleration/deceleration is found through solving an energy-optimal control problem. 
Moreover, the relative position of CAV $2$ and CAV $3$ with the conflict tuple $\mathcal{C}_{3,2}=[1,11,2,12]$ is shown in Fig. \ref{fig:6} for zones which rear-end collision may occur. 
\begin{table}[ht]\label{table}
	\centering 
	\caption{Schedules and release time of first two zones of CAV 2 and CAV 3.}
	\begin{tabular}{lccc} \label{tab:experimentalResults}
		CAV Index & First Zone & $R_i^3$(s) & $T_i^3$(s) \\ \hline
		\vspace{0.25mm}
		2&14&14.54&14.54\\
		3&10 &15.04 &16.04
	\end{tabular}
\end{table}

\begin{figure}
\includegraphics[width=8.5cm]{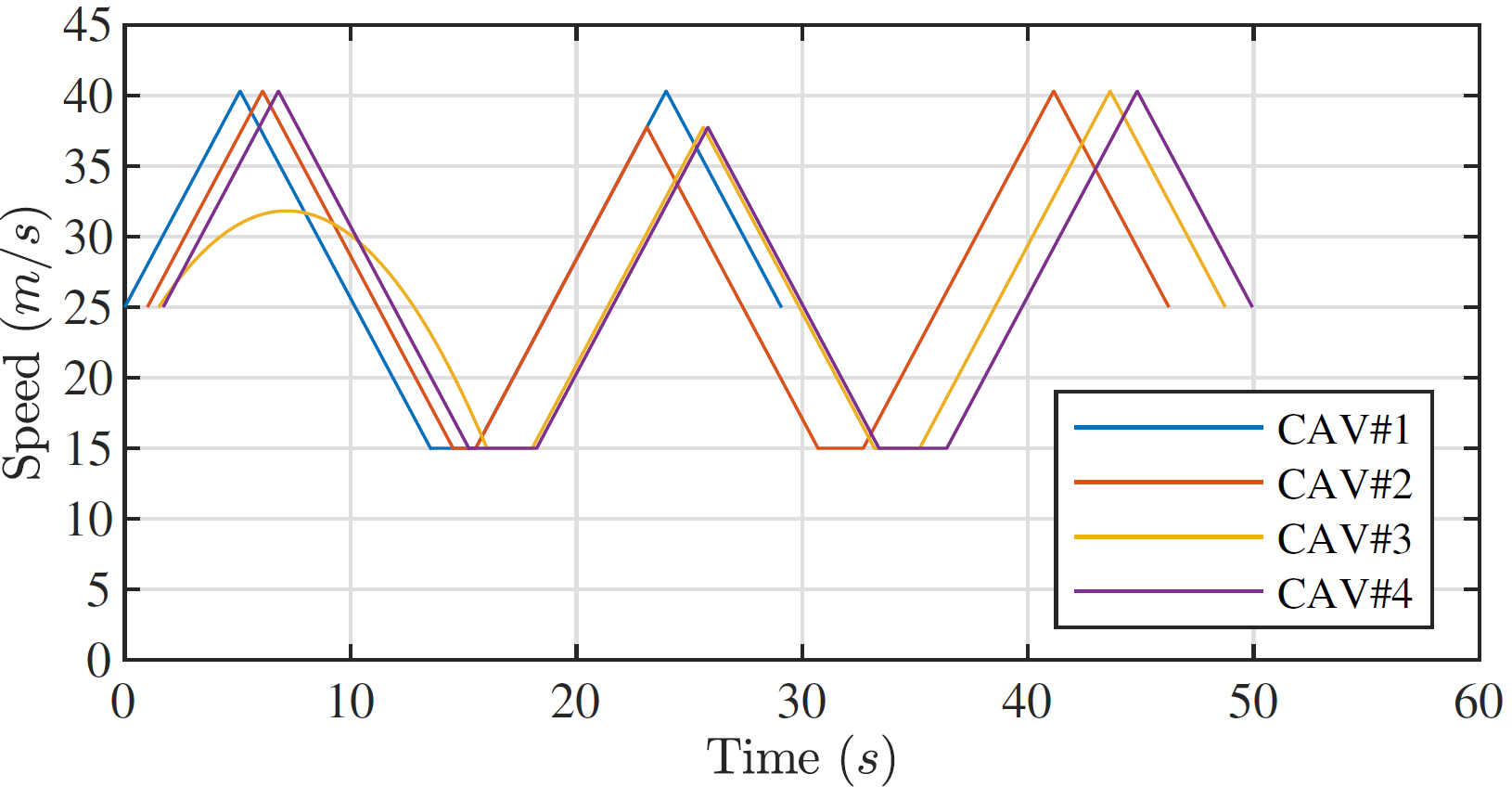}
\caption{Speed profile for the first 4 CAVs.}
\label{fig:5}%
\end{figure}
\begin{figure}
\includegraphics[width=8.5cm]{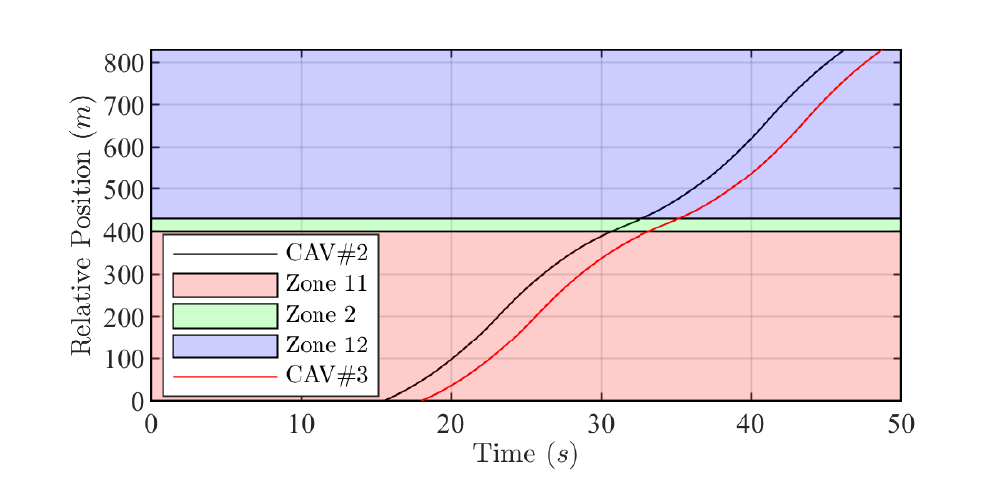}
\caption{Relative Position of CAV\#2 and CAV\#3 at zones which rear-end collision may occur.}
\label{fig:6}%
\end{figure}

\section{Concluding Remarks and future}
In this paper, we proposed a decentralized optimal control framework for CAVs for two interconnected intersections using scheduling theory. We formulated the control problem and provided a solution that can be implemented in real time. Upon entering the control zones, CAV solves a scheduling problem; however, if there are multiple CAVs entering the control zone at the same time, the scheduling problem is solved sequentially with their order in the queue.  The solution yields the optimal acceleration/deceleration of each CAV under the safety constraint at conflict zones.  Our objective is to minimize the travel time for each CAV on its path. If there is no such feasible solution, then each CAV solves an energy optimal control problem. 
There two potential directions for the future research: (1) to involve fully constrained energy-optimal and time-optimal problems as well as to investigate safety constraints on each zone, and explore when a safety constraint becomes active; and (2) to address uncertainty in data communication and control process. %
\bibliographystyle{IEEEtran}
\bibliography{ref}

\begin{thebibliography}{10}
\providecommand{\url}[1]{#1}
\csname url@samestyle\endcsname
\providecommand{\newblock}{\relax}
\providecommand{\bibinfo}[2]{#2}
\providecommand{\BIBentrySTDinterwordspacing}{\spaceskip=0pt\relax}
\providecommand{\BIBentryALTinterwordstretchfactor}{4}
\providecommand{\BIBentryALTinterwordspacing}{\spaceskip=\fontdimen2\font plus
\BIBentryALTinterwordstretchfactor\fontdimen3\font minus
  \fontdimen4\font\relax}
\providecommand{\BIBforeignlanguage}[2]{{%
\expandafter\ifx\csname l@#1\endcsname\relax
\typeout{** WARNING: IEEEtran.bst: No hyphenation pattern has been}%
\typeout{** loaded for the language `#1'. Using the pattern for}%
\typeout{** the default language instead.}%
\else
\language=\csname l@#1\endcsname
\fi
#2}}
\providecommand{\BIBdecl}{\relax}
\BIBdecl

\bibitem{Cassandras2017}
C.~G. Cassandras, ``{Automating mobility in smart cities},'' \emph{Annual
  Reviews in Control}, 2017.

\bibitem{schrank20152015}
D.~Schrank, B.~Eisele, T.~Lomax, and J.~Bak, ``2015 urban mobility scorecard,''
  2015.

\bibitem{Klein2016a}
I.~Klein and E.~Ben-Elia, ``{Emergence of cooperation in congested road
  networks using ICT and future and emerging technologies: A game-based
  review},'' \emph{Transportation Research Part C: Emerging Technologies},
  vol.~72, pp. 10--28, 2016.

\bibitem{Melo2017a}
S.~Melo, J.~Macedo, and P.~Baptista, ``{Guiding cities to pursue a smart
  mobility paradigm: An example from vehicle routing guidance and its traffic
  and operational effects},'' \emph{Research in Transportation Economics},
  vol.~65, pp. 24--33, 2017.

\bibitem{Lioris2017}
J.~Lioris, R.~Pedarsani, F.~Y. Tascikaraoglu, and P.~Varaiya, ``{Platoons of
  connected vehicles can double throughput in urban roads},''
  \emph{Transportation Research Part C: Emerging Technologies}, 2017.

\bibitem{Rios-Torres2015}
J.~Rios-Torres, A.~A. Malikopoulos, and P.~Pisu, ``{Online Optimal Control of
  Connected Vehicles for Efficient Traffic Flow at Merging Roads},'' in
  \emph{2015 IEEE 18th International Conference on Intelligent Transportation
  Systems}, 2015, pp. 2432--2437.

\bibitem{Rios-Torres2}
J.~Rios-Torres and A.~A. Malikopoulos, ``{Automated and Cooperative Vehicle
  Merging at Highway On-Ramps},'' \emph{IEEE Transactions on Intelligent
  Transportation Systems}, vol.~18, no.~4, pp. 780--789, 2017.

\bibitem{Ntousakis:2016aa}
I.~A. Ntousakis, I.~K. Nikolos, and M.~Papageorgiou, ``Optimal vehicle
  trajectory planning in the context of cooperative merging on highways,''
  \emph{Transportation Research Part C: Emerging Technologies}, vol.~71, pp.
  464--488, 2016.

\bibitem{Malikopoulos2018a}
L.~Zhao, A.~A. Malikopoulos, and J.~Rios-Torres, ``Optimal control of connected
  and automated vehicles at roundabouts: An investigation in a mixed-traffic
  environment,'' in \emph{15th IFAC Symposium on Control in Transportation
  Systems}, 2018, pp. 73--78.

\bibitem{Zhang2016a}
Y.~Zhang, A.~A. Malikopoulos, and C.~G. Cassandras, ``Optimal control and
  coordination of connected and automated vehicles at urban traffic
  intersections,'' in \emph{Proceedings of the American Control Conference},
  2016, pp. 6227--6232.

\bibitem{Zhang:2017aa}
------, ``Decentralized optimal control for connected automated vehicles at
  intersections including left and right turns,'' in \emph{2017 IEEE 56th
  Annual Conference on Decision and Control (CDC)}, 2017, pp. 4428--4433.

\bibitem{Malikopoulos2018b}
A.~Stager, L.~Bhan, A.~A. Malikopoulos, and L.~Zhao, ``A scaled smart city for
  experimental validation of connected and automated vehicles,'' in \emph{15th
  IFAC Symposium on Control in Transportation Systems}, 2018, pp. 130--135.

\bibitem{beaver2019demonstration}
L.~E. Beaver, B.~Chalaki, A.~Mahbub, L.~Zhao, R.~Zayas, and A.~A. Malikopoulos,
  ``Demonstration of a time-efficient mobility system using a scaled smart
  city,'' \emph{arXiv preprint arXiv:1903.01632}, 2019.

\bibitem{Malikopoulos2017}
A.~A. Malikopoulos, C.~G. Cassandras, and Y.~Zhang, ``A decentralized
  energy-optimal control framework for connected automated vehicles at
  signal-free intersections,'' \emph{Automatica}, vol.~93, pp. 244--256, 2018.

\bibitem{Malikopoulos2018c}
A.~A. Malikopoulos, S.~Hong, B.~Park, J.~Lee, and S.~Ryu, ``Optimal control for
  speed harmonization of automated vehicles,'' \emph{IEEE Transactions on
  Intelligent Transportation Systems}, 2018 (in press).

\bibitem{Zhao2018}
L.~Zhao and A.~A. Malikopoulos, ``{Decentralized Optimal Control of Connected
  and Automated Vehicles in a Corridor},'' in \emph{2018 21st International
  Conference on Intelligent Transportation Systems (ITSC)}, 2018, pp.
  1252--1257.

\bibitem{Rios2018}
J.~Rios-Torres and A.~A. Malikopoulos, ``{Impact of Partial Penetrations of
  Connected and Automated Vehicles on Fuel Consumption and Traffic Flow},''
  \emph{IEEE Transactions on Intelligent Vehicles}, vol.~3, no.~4, pp.
  453--462, 2018.

\bibitem{DeCampos2015a}
G.~R. {De Campos}, F.~{Della Rossa}, and A.~Colombo, ``{Optimal and least
  restrictive supervisory control: Safety verification methods for human-driven
  vehicles at traffic intersections},'' \emph{Proceedings of the IEEE
  Conference on Decision and Control}, vol. 54rd IEEE, pp. 1707--1712, 2015.

\bibitem{Colombo2015}
A.~Colombo and D.~{Del Vecchio}, ``{Least restrictive supervisors for
  intersection collision avoidance: A scheduling approach},'' \emph{IEEE
  Transactions on Automatic Control}, 2015.

\bibitem{Ahn2014}
H.~Ahn, A.~Colombo, and D.~{Del Vecchio}, ``{Supervisory control for
  intersection collision avoidance in the presence of uncontrolled vehicles},''
  in \emph{Proceedings of the American Control Conference}, 2014.

\bibitem{Ahn2016}
H.~Ahn and D.~{Del Vecchio}, ``{Semi-autonomous Intersection Collision
  Avoidance through Job-shop Scheduling},'' \emph{Proceedings of the 19th
  International Conference on Hybrid Systems: Computation and Control - HSCC
  '16}, pp. 185--194, 2016.

\bibitem{mahbub2019energy}
A.~Mahbub, L.~Zhao, D.~Assanis, and A.~A. Malikopoulos, ``Energy-optimal
  coordination of connected and automated vehicles at multiple intersections,''
  in \emph{2019 American Control Conference}, 2019.

\bibitem{ross2015primer}
I.~M. Ross, \emph{A primer on Pontryagin's principle in optimal control}.\hskip
  1em plus 0.5em minus 0.4em\relax Collegiate publishers, 2015.

\bibitem{pinedo2016scheduling}
M.~L. Pinedo, \emph{Scheduling: theory, algorithms, and systems}.\hskip 1em
  plus 0.5em minus 0.4em\relax Springer, 2016.

\bibitem{Colombo}
A.~Colombo and D.~Del~Vecchio, ``Efficient algorithms for collision avoidance
  at intersections,'' in \emph{Proceedings of the 15th ACM international
  conference on Hybrid Systems: Computation and Control}.\hskip 1em plus 0.5em
  minus 0.4em\relax ACM, 2012, pp. 145--154.

\bibitem{malikopoulos2010online}
A.~A. Malikopoulos, P.~Y. Papalambros, and D.~N. Assanis, ``Online
  identification and stochastic control for autonomous internal combustion
  engines,'' \emph{Journal of dynamic systems, measurement, and control}, vol.
  132, no.~2, p. 024504, 2010.

\bibitem{pontryagin2018mathematical}
L.~S. Pontryagin, \emph{Mathematical theory of optimal processes}.\hskip 1em
  plus 0.5em minus 0.4em\relax Routledge, 2018.

\end{thebibliography}

\end{document}